\documentclass[a4paper, 12pt]{amsart}
\usepackage{mathrsfs}
\usepackage{amsmath}
\usepackage{amssymb}
\usepackage{verbatim}
\usepackage{xcolor}
\usepackage{CJK}
\usepackage{graphicx}
\usepackage[all]{xy}
\usepackage{appendix}
\DeclareGraphicsExtensions{.eps,.ps,.jpg,.bmp}

\theoremstyle{plain}

\newtheorem{theorem}{Theorem}[section]
\newtheorem{proposition}{Proposition}
\newtheorem{corollary}{Corollary}[section]

\newtheorem{lemma}{Lemma}[section]

\newtheorem{example}{Example}[section]

\theoremstyle{definition}
\newtheorem{definition}{Definition}[section]

\theoremstyle{remark}
\newtheorem{remark}{Remark}[section]
\newtheorem*{acknowledgment}{Acknowledgment}

\title{Nefness of the direct images of relative canonical bundles}
\author{Jingcao Wu\\ \emph{School of Mathematical Sciences, Fudan University}}

\begin{document}

\begin{abstract}
Given a fibration $f$ between two projective manifolds $X$ and $Y$, we discuss the nefness of the direct images $f_{\ast}(K_{X/Y}\otimes L)$, where $(L,h)$ is a pseudo-effective line bundle with mild singularity.
\end{abstract}
\maketitle
\pagestyle{plain}

\section{Introduction}
Assume that $f:X\rightarrow Y$ is a surjective fibration, i.e. a morphism with connected fibres, between two projective manifolds $X$ and $Y$. $L$ is a line bundle on $X$. The positivity of the associated direct image $f_{\ast}(K_{X/Y}\otimes L)$ is of much importance to understand the geometry of this fibration. There are fruitful results in this subject, such as \cite{Ber08,Ber09,Hor10,Kaw81,Kaw82,Ko86a,Ko86b,Ko87,Vie82,Vie83}. It turns out that the positivity of $f_{\ast}(K_{X/Y}\otimes L)$ is deeply influenced by $L$.

In this paper, we focus on a general pseudo-effective line bundle $(L,h)$. First, we prove the following Koll\'{a}r-type vanishing theorem.
\begin{theorem}\label{t1}
Let $f:X\rightarrow Y$ be a surjective fibration between projective manifolds $X$ and $Y$. Let $(L,h)$ be a pseudo-effective line bundle over $X$. If $A$ is an ample line bundle over $Y$, then for any $i>0$ and $q\geqslant0$,
\begin{equation*}
    H^{i}(Y,R^{q}f_{\ast}(K_{X}\otimes L\otimes\mathscr{I}(h))\otimes A)=0.
\end{equation*}
Here and in the rest of this paper, $\mathscr{I}(h)$ refers to the multiplier ideal sheaf associated to $h$.
\end{theorem}

This result is a generalization of the main result in \cite{Wu17}. A direct consequence is
\begin{corollary}\label{c1}
Under the same assumptions as in Theorem \ref{t1}, if $A$ is an ample and globally generated line bundle and $A^{\prime}$ is a nef line bundle over $Y$, then the sheaf $R^{q}f_{\ast}(K_{X}\otimes L\otimes\mathscr{I}(h))\otimes A^{m}\otimes A^{\prime}$ is globally generated for any $q\geqslant0$ and $m\geqslant\dim Y+1$.
\end{corollary}

Our next result is as follows:
\begin{theorem}\label{t2}
Let $f:X\rightarrow Y$ be a smooth fibration between projective manifolds $X$ and $Y$. Let $(L,h)$ be a pseudo-effective line bundle over $X$. Assume that
\begin{equation*}
    \psi_{l}(y)=(\int_{X_{y}}|u|^{2}_{h})^{l}\in L^{1}_{\mathrm{loc}}(Y)
\end{equation*}
for any $l\in\mathbb{Z}_{+}$ and local section $u$ of $f_{\ast}(K_{X/Y}\otimes L\otimes\mathscr{I}(h))$. Then if $\mathcal{E}=f_{\ast}(K_{X/Y}\otimes L\otimes\mathscr{I}(h))$ is locally free, it is nef.
\end{theorem}

Since $u$ is a local section of $f_{\ast}(K_{X/Y}\otimes L\otimes\mathscr{I}(h))$, $\psi_{l}$ is not identically equal to $\infty$, so it is meaningful to ask it to be locally integrable.

The assumption that $\mathcal{E}$ is locally free is not essential. First, it is proved in \cite{Mat16} that $\mathcal{E}$ is always torsion-free. On the other hand, most notions concerning the positivity can be generalized from a locally free sheaf to a torsion-free one \cite{Pau16,PT14}. Technically, a torsion-free coherent sheaf $\mathcal{E}$ must be locally free outside a $2$-codimensional subvariety. So $\mathcal{E}$ is called ample (resp. nef, big,...) if it is ample (resp. nef, big,...) on its locally free part. The proof for Theorem \ref{t2} also works on its locally free part when $\mathcal{E}$ is merely torsion-free. Moreover, through this reduction, we will can treat $f_{\ast}(K_{X/Y}\otimes L)$ and $f_{\ast}(K_{X/Y}\otimes L\otimes\mathscr{I}(h))$ as locally free sheaves in the rest of this paper.

There are various articles studying the nefness of the direct images. One refers to \cite{Ber08,Pau16,PT14} for more details. The main stream of these articles is to use the Ohsawa--Takegoshi extension theorem to locally extend the sections along the certain fibre of $f:X\rightarrow Y$, which asks that the singular metric $h$ of $L$ to be well-defined along this fibre. As a result, if one wants to deduce the nefness of the direct images from their work, $h$ is required to be well-defined along all of the fibres. In this paper, it will be more flexible. Indeed, it is easy to verify that there exists such a $\psi_{l}$, which is locally integrable, but $h$ needs not to be well-defined on every $X_{y}$. We will show two examples in the text. As a consequence, we can prove that
\begin{theorem}\label{t3}
Let $E$ be a stable vector bundle over a compact Riemann surface $Y$. Assume that $(\det E,\phi)$ is pseudo-effective. If the Lelong number $\nu(l\phi)=0$ for all positive integer $l$, $E$ is nef.
\end{theorem}

At last, we will present in the appendix a general version of the positivity of the direct images appeared in \cite{Ber08,Pau16,PT14} in order to distinguish our result.

The organization of this paper is as follows. We first introduce some notions, including the definition of the positivity of a vector bundle, the fibre product and so on. Then we proceed to a talk about the vanishing theorem of the direct images in \S3. The further discussions about the positivity will be given in \S4. Two examples including the proof of Theorem \ref{t3} are presented in \S5. Finally, for reader's convenience and the completeness of this paper, we will provide in Appendix a general version of the nefness result of the direct images proved via the Ohsawa--Takagoshi theorem.

\section{Preliminaries}
Let $E$ be a holomorphic vector bundle of rank $r$ over a projective manifold $Y$. We denote by $\mathbb{P}(E^{\ast})$ the projectivized bundle of $E^{\ast}$ and by $\mathcal{O}_{E}(1):=\mathcal{O}_{\mathbb{P}(E^{\ast})}(1)$ the tautological line bundle. Let $\pi:X:=\mathbb{P}(E^{\ast})\rightarrow Y$ be the canonical projection. First we recall some notions concerning the positivity of $E$.

To begin, we discuss a relatively standard notion of a singular metric on a vector bundle introduced in \cite{Rau15}. By this we basically mean a measurable map $h$ from the base $Y$ to the space of non-negative Hermitian forms on the fiber. Moreover, $\det h<\infty$ almost everywhere. This additional condition helps to define the curvature associated to $h$.

We do, however, need to allow the Hermitian form to take the value $\infty$ for some vectors at some points in the base. Given a singular metric $h$, the norm function $|u|^{2}_{h}$, for any holomorphic section $u$ of $E$, is a measurable function from the total space of $E$ to $[0,\infty]$, whose restriction to any fiber $E_{y}$ is a quadratic form. (Certainly $|u|^{2}_{h}$ may be $\infty$ on a subspace of $E_{y}$.)

\begin{definition}
Let $h$ be a (singular) Hermitian metric of $E$.
\begin{enumerate}
  \item $h$ is negatively curved (or $h$ has Griffiths semi-negative curvature), if for any open subset $U\subset Y$ and any $v\in H^{0}(U,E)$, $\log|v|^{2}_{h}$ is plurisubharmonic on $U$.
  \item $h$ is positively curved (or $h$ has Griffiths semi-positive curvature), if the dual singular Hermitian metric $h^{\ast}$ of the dual vector bundle $E^{\ast}$ is negatively curved.
  \item $E$ is weakly positive in the sense of Viehweg, if on some Zariski open subset $U\subset Y$, for any integer $a>0$, there exists an integer $b>0$ such that $S^{ab}(E)\otimes bA$ is generated by global sections on $U$. Here $A$ is an auxiliary ample line bundle over $Y$.
  \item $E$ is pseudo-effective, if $\mathcal{O}_{E}(1)$ is pseudo-effective and the image of the non-nef locus $\text{NNef}(\mathcal{O}_{E}(1))$ (i.e. the union of all curves $C$ on $X$ with $\mathcal{O}_{E}(1)\cdot C<0$) under $\pi$ is a proper subset of $Y$.
  \item $E$ is almost nef if there exists a countable family of proper subvarieties $Z_{i}$ of $Y$ such that $E|_{C}$ is nef for any curve $C\not\subset\cup_{i}Z_{i}$.
  \item $E$ is nef, if $\mathcal{O}_{E}(1)$ is nef.
\end{enumerate}
\end{definition}

\begin{remark}
The relationships among the notions above are summarized below.

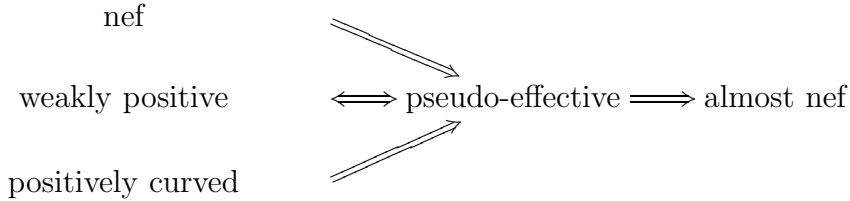
\xymatrix@R=0.5cm{
\textrm{nef} &  \ar@{=>}[dr]  &    &                       \\
\textrm{weakly positive} & \ar@{<=>}[r] & \textrm{pseudo-effective} \ar@{=>}[r] & \textrm{almost nef}\\
\textrm{positively curved} & \ar@{=>}[ur]  &    &}

Moreover, these notions can be generalized to a torsion-free coherent sheaf $E$. Technically, $E$ is locally free outside a 2-codimensional subvariety $Z$, then $E$ is nef (resp. pseudo-effective, almost nef,...) if $E|_{Y-Z}$ is. One can refer to \cite{DPS94,DPS01,Pau16,PT14} for the more details.
\end{remark}

Next we recall the definition of the fibre product.

\begin{definition}\label{d1}
Let $f:X\rightarrow Y$ be a fibration between two projective manifolds $X$ and $Y$, the fibre product, denoted by $(X\times_{Y}X,p^{2}_{1},p^{2}_{2})$, is a projective manifold coupled with two morphisms (we will also refer the fibre product merely to the manifold $X\times_{Y}X$ itself if nothing is confused), which satisfies the following properties:

1.The diagram
\begin{equation*}
\begin{array}[c]{ccc}
X\times_{Y}X&\stackrel{p^{2}_{2}}{\rightarrow}&X\\
\scriptstyle{p^{2}_{1}}\downarrow&&\downarrow\scriptstyle{f}\\
X&\stackrel{f}{\rightarrow}&Y
\end{array}
\end{equation*}
commutes.

2.If there is another projective manifold $Z$ with morphisms $q_{1}, q_{2}$ such that the diagram
\begin{equation*}
\begin{array}[c]{ccc}
Z&\stackrel{q_{2}}{\rightarrow}&X\\
\scriptstyle{q_{1}}\downarrow&&\downarrow\scriptstyle{f}\\
X&\stackrel{f}{\rightarrow}&Y
\end{array}
\end{equation*}
commutes, then there must exist a unique $g:Z\rightarrow X\times_{Y}X$ such that $p^{2}_{1}\circ g=q_{1},p^{2}_{2}\circ g=q_{2}$.

We inductively define the $m$-fold fibre product, and denote it by $X\times_{Y}\cdot\cdot\cdot\times_{Y}X$. Moreover, we denote two projections by
\begin{equation*}
 p^{m}_{1}:X\times_{Y}\cdot\cdot\cdot\times_{Y}X\rightarrow X
\end{equation*}
and
\begin{equation*}
 p^{m}_{2}:\underbrace{X\times_{Y}\cdot\cdot\cdot\times_{Y}X}_{m}\rightarrow \underbrace{X\times_{Y}\cdot\cdot\cdot\times_{Y}X}_{m-1}
\end{equation*}
respectively.
\end{definition}

The meaning of the fibre product is clear in the view of geometry. In fact, when we do the fibre product, geometrically it just means that we take the Cartesian product along the fibre. Namely, if $y$ is a regular value of $f$,
\begin{equation*}
 (X\times_{Y}\cdot\cdot\cdot\times_{Y}X)_{y}=X_{y}\times\cdot\cdot\cdot\times X_{y}.
\end{equation*}

Finally we introduce the following two lemmas without proof for the later use.

\begin{lemma}(Projection formula)\label{l1}
If $f:X\rightarrow Y$ is a holomorphic morphism between two complex manifolds $X$ and $Y$, $\mathcal{F}$ is a coherent sheaf on $X$, and $\mathcal{E}$ is a locally free sheaf on $Y$, then there is a natural isomorphism
\begin{equation*}
  f_{\ast}(\mathcal{F}\otimes f^{\ast}\mathcal{E})\cong f_{\ast}(\mathcal{F})\otimes\mathcal{E}.
\end{equation*}
\end{lemma}

\begin{lemma}(Base change)\label{l2}
Assume that $f:X\rightarrow Y, v:X^{\prime}\rightarrow X$ and $g:X^{\prime}\rightarrow Y^{\prime}$ are holomorphic morphisms between complex manifolds $X,X^{\prime},Y$ and $Y^{\prime}$. Let $\mathcal{F}$ be a coherent sheaf on $X$. $u:Y^{\prime}\rightarrow Y$ is a smooth morphism, such that
\begin{equation*}
\begin{array}[c]{ccc}
X^{\prime}&\stackrel{v}{\rightarrow}&X\\
\scriptstyle{g}\downarrow&&\downarrow\scriptstyle{f}\\
Y^{\prime}&\stackrel{u}{\rightarrow}&Y
\end{array}
\end{equation*}
commutes. Then for all $q\geqslant0$ there is a natural isomorphism
\begin{equation*}
   u^{\ast}R^{q}f_{\ast}(\mathcal{F})\cong R^{q}g_{\ast}(v^{\ast}\mathcal{F}).
\end{equation*}
\end{lemma}

\section{A Koll\'{a}r-type vanishing theorem}
In this section, we shall prove Theorem \ref{t1}. As is well-known to the experts,  Koll\'{a}r's vanihisng theorem comes from two things: the torsion-freeness of the (higher) direct images and Koll\'{a}r's injectivity theorem. Fortunately, they have been generalized in \cite{GoM17,Mat16} to be suitable for our situation.

The Koll\'{a}r-type torsion-freeness says that
\begin{theorem}(\cite{Mat16})\label{t4}
Let $\pi:X\rightarrow\Delta$ be a surjective proper K\"{a}hler morphism from a complex manifold $X$ to an analytic space $\Delta$, and $(F,h)$ be a (singular) Hermitian line bundle over $X$ with semi-positive curvature. Then $R^{q}\pi_{\ast}(K_{X/Y}\otimes F\otimes\mathscr{I}(h))$ is torsion-free for every $q$.
\end{theorem}

The injectivity theorem is
\begin{theorem}(Gongyo--Matsumura, \cite{GoM17}) \label{t5}
Let $(F,h_{F})$ and $(L,h_{L})$ be (singular) Hermitian line bundles with semi-positive curvature on a compact K\"{a}hler manifold $X$. Assume that there exists an $\mathbb{R}$-effective divisor $\Delta$ on $X$ such that $h_{F}=h^{a}_{L}h_{\Delta}$ for a positive real number $a$ and the singular metric $h_{\Delta}$ defined by $\Delta$. Then for a section $s$ of $L$ satisfying $\sup_{X}|s|_{h_{L}}<\infty$, the multiplication map induced by $s$
\begin{equation*}
   \Phi_{s}:H^{q}(X,K_{X}\otimes F\otimes\mathscr{I}(h_{F}))\stackrel{\otimes s}\rightarrow H^{q}(X,K_{X}\otimes F\otimes L\otimes\mathscr{I}(h_{F}h_{L}))
\end{equation*}
is injective for any $q$.
\end{theorem}

Possessing these two theorems, the proof of Theorem \ref{t1} is routine.
\begin{proof}[Proof of Theorem \ref{t1}]
By asymptotic Serre vanishing theorem, we can choose a positive integer $m_{0}$ such that for all $m\geqslant m_{0}$,
\begin{equation*}
    H^{i}(Y,R^{q}f_{\ast}(K_{X}\otimes L\otimes\mathscr{I}(h))\otimes mA)=0
\end{equation*}
for $i>0, q\geqslant0$.
Fix an integer $m$ such that $m\geqslant m_{0}$ and  $mA$ is very ample.

We prove the theorem by induction on $n=\dim X$, the case $n=0$ being trivial.
Denote $A'=f^\ast (A)$ and let $H^{\prime}\in|mA^{\prime}|$ be the pullback of a general divisor $H\in|mA|$.
It follows from Bertini's theorem that we can assume $H$ is integral and $H^{\prime}$ is smooth (though possibly disconnected). Then we have a short exact sequence
\begin{equation*}
\begin{split}
0&\rightarrow K_{X}\otimes L\otimes A^{\prime}\rightarrow K_{X}\otimes L\otimes(m+1)A^{\prime}\\
&\rightarrow K_{H^{\prime}}\otimes L\otimes A^{\prime}|_{H^{\prime}}\rightarrow0
\end{split}
\end{equation*}
induced by multiplication with a section defining $H^{\prime}$. We claim that
\begin{equation}\label{e1}
\begin{split}
0&\rightarrow K_{X}\otimes L\otimes\mathscr{I}(h)\otimes A^{\prime}\rightarrow K_{X}\otimes L\otimes\mathscr{I}(h)\otimes(m+1)A^{\prime}\\
&\rightarrow K_{H^{\prime}}\otimes L\otimes\mathscr{I}(h)\otimes A^{\prime}|_{H^{\prime}}\rightarrow0
\end{split}
\end{equation}
is also an exact sequence. We only need to verify that
\begin{equation*}
    i:K_{X}\otimes L\otimes\mathscr{I}(h)\otimes A^{\prime}\rightarrow K_{X}\otimes L\otimes\mathscr{I}(h)\otimes(m+1)A^{\prime}
\end{equation*}
is injective. Consider the following communicative diagram
\begin{equation*}
\xymatrix{
    K_{X}\otimes L\otimes\mathscr{I}(h)\otimes A^{\prime}\ar[d]^{j} & \stackrel{i}{\longrightarrow} & K_{X}\otimes L\otimes\mathscr{I}(h)\otimes(m+1)A^{\prime}\ar[d]^{k} \\
    K_{X}\otimes L\otimes A^{\prime} & \stackrel{l}{\longrightarrow} &K_{X}\otimes L\otimes(m+1)A^{\prime}.
    }
\end{equation*}
Since $l\circ j=k\circ i$, the injectivity of $i$ follows from the injectivity of $j$ and $l$, which is obvious. We get from the short exact sequence (\ref{e1}) a long exact sequence
\begin{equation}\label{e2}
\begin{split}
0&\rightarrow f_{\ast}(K_{X}\otimes L\otimes\mathscr{I}(h)\otimes A^{\prime})\rightarrow f_{\ast}(K_{X}\otimes L\otimes\mathscr{I}(h)\otimes(m+1)A^{\prime})\\
&\rightarrow f_{\ast}(K_{H^{\prime}}\otimes L\otimes\mathscr{I}(h)\otimes A^{\prime}|_{H^{\prime}})\rightarrow R^{1}f_{\ast}(K_{X}\otimes L\otimes\mathscr{I}(h)\otimes A^{\prime})\\
&\rightarrow R^{1}f_{\ast}(K_{X}\otimes L\otimes\mathscr{I}(h)\otimes(m+1)A^{\prime})\rightarrow\cdots
\end{split}
\end{equation}
By Theorem \ref{t4} all the higher direct images of $K_{X}\otimes L\otimes\mathscr{I}(h)\otimes A^{\prime}$ are torsion-free.
Clearly the sheaves $R^{q}f_{\ast}(K_{H^{\prime}}\otimes L\otimes\mathscr{I}(h)\otimes A^{\prime}|_{H^{\prime}})$ are torsion on $H$.
Hence the long exact sequence (\ref{e2}) can be split into a family of short exact sequences:  for all $q\geq 0$,
\begin{equation}\label{e3}
\begin{split}
0&\rightarrow R^{q}f_{\ast}(K_{X}\otimes L\otimes\mathscr{I}(h)\otimes A^{\prime})\rightarrow R^{q}f_{\ast}(K_{X}\otimes L\otimes\mathscr{I}(h)\otimes(m+1)A^{\prime})\\
&\rightarrow R^{q}f_{\ast}(K_{H^{\prime}}\otimes L\otimes\mathscr{I}(h)\otimes A^{\prime}|_{H^{\prime}})\rightarrow0.
\end{split}
\end{equation}
On the other hand, applying the inductive hypothesis to each connected component of $H^{\prime}$,
we conclude that for all $i\geqslant1$
\begin{equation*}
   H^{i}(Y,R^{q}f_{\ast}(K_{H^{\prime}}\otimes L\otimes\mathscr{I}(h)\otimes A^{\prime}|_{H^{\prime}}))=0.
\end{equation*}
Furthermore, by the choice of $m$ we also have for all $i\geqslant1$
\begin{equation}\label{e4}
    H^{i}(Y,R^{q}f_{\ast}(K_{X}\otimes L\otimes\mathscr{I}(h)\otimes(m+1)A^{\prime}))=0.
\end{equation}
Now by taking the long exact sequence from the short exact sequence (\ref{e3}), we find for every $i>1$
\begin{equation*}
   H^{i}(Y,R^{q}f_{\ast}(K_{X}\otimes L\otimes\mathscr{I}(h)\otimes A^{\prime}))=0.
\end{equation*}
This proves the theorem for the cases where $i>1$.

To prove the case where $i=1$, we denote
\begin{equation*}
   B_{l}:=H^{1}(Y,R^{q}f_{\ast}(K_{X}\otimes L\otimes\mathscr{I}(h)\otimes lA^{\prime})).
\end{equation*}
We have $B_{m+1}=0$. Hence we consider the following commutative diagram.
\begin{equation*}
\xymatrix{
    B_{1}\ar[d] & \stackrel{\phi}{\longrightarrow} & H^{q+1}(X,K_{X}\otimes L\otimes\mathscr{I}(h)\otimes A^{\prime})\ar[d]^{\psi} \\
    B_{m+1} & {\longrightarrow} &H^{q+1}(X,K_{X}\otimes L\otimes\mathscr{I}(h)\otimes(m+1)A^{\prime})
    }
\end{equation*}
Here the horizontal maps are the canonical injective maps coming out of the Leray spectral sequence, and the vertical maps are induced by multiplication with sections defining $H^{\prime}$ and $H$ respectively.
By Theorem \ref{t5} the map $\psi$ is injective, and hence the composition $\psi\circ\phi$ is also injective.
Hence $B_1=0$ and we finish the proof of the theorem for the case where $i=1$.
\end{proof}

Before proving Corollary \ref{c1}, we will review the definition and a basic result of the Castelnuovo--Mumford regularity \cite{Mum66}.

\begin{definition}
Let $Y$ be a projective manifold and $L$  an ample and globally generated line bundle over $X$. Given an integer $m$, a coherent sheaf $F$ on $Y$ is $m$-regular with respect to $L$ if for all $i\geqslant1$
\begin{equation*}
H^{i}(Y,F\otimes L^{m-i})=0.
\end{equation*}
\end{definition}

\begin{theorem}(Mumford, \cite{Mum66})\label{t6}
Let $Y$ be a projective manifold and $L$ an ample and globally generated line bundle over $Y$. If $F$ is a coherent sheaf on $Y$ that is $m$-regular with respect to $L$, then the sheaf $F\otimes L^{m}$ is globally generated.
\end{theorem}

After this, we can  prove Corollary \ref{c1}.
\begin{proof}[Proof of Corollary \ref{c1}]
It follows from Theorem \ref{t1} that for every $i\geqslant1$ and $m\geqslant\dim Y+1$,
\begin{equation*}
  H^{i}(Y,R^{q}f_{\ast}(K_{X}\otimes L\otimes\mathscr{I}(h))\otimes A^{m-i}\otimes A^{\prime})=0.
\end{equation*}
Hence the sheaf $R^{q}f_{\ast}(K_{X}\otimes L\otimes\mathscr{I}(h))\otimes A^{m}\otimes A^{\prime}$ is $0$-regular with respect to $A$.
So it is globally generated by Theorem \ref{t6}.
\end{proof}

\section{The positivity of $f_{\ast}(K_{X/Y}\otimes L)$}
In this section, we shall prove Theorem \ref{t2}. The ingredient is the following observation:
\begin{lemma}\label{l3}
For any positive integer $m$, consider the $m$-fold fibre product $f_{m}:X_{m}=X\times_{Y}\cdot\cdot\cdot\times_{Y}X\rightarrow Y$. With the same notations as in Definition \ref{d1}, we have
\begin{equation*}
\begin{split}
  &(f_{m})_{\ast}(K_{X_{m}/Y}\otimes(p^{m}_{1}\otimes p^{m-1}_{1}p^{m}_{2}\otimes p^{m-2}_{1}p^{m-1}_{2}p^{m}_{2}\otimes\cdot\cdot\cdot\otimes p^{1}_{2}p^{2}_{2}...p^{m}_{2})^{\ast}(L))\\
  =&f_{\ast}(K_{X/Y}\otimes L)^{\otimes m}.
\end{split}
\end{equation*}
Here $L$ is an arbitrary line bundle on $X$. Moreover, if $h$ is a (singular) metric of $L$, $h_{m}=(p^{m}_{1}\otimes p^{m-1}_{1}p^{m}_{2}\otimes p^{m-2}_{1}p^{m-1}_{2}p^{m}_{2}\otimes\cdot\cdot\cdot\otimes p^{1}_{2}p^{2}_{2}...p^{m}_{2})^{\ast}h$ will be a metric of
\begin{equation*}
    L_{m}:=(p^{m}_{1}\otimes p^{m-1}_{1}p^{m}_{2}\otimes p^{m-2}_{1}p^{m-1}_{2}p^{m}_{2}\otimes\cdot\cdot\cdot\otimes p^{1}_{2}p^{2}_{2}...p^{m}_{2})^{\ast}(L).
\end{equation*}
We have the following subadditivity property of the multiplier ideal sheaves:
\begin{equation*}
    \mathscr{I}(h_{m})\subset(p^{m}_{1}\otimes p^{m-1}_{1}p^{m}_{2}\otimes p^{m}_{1}\otimes p^{m-2}_{1}p^{m-1}_{2}p^{m}_{2}\otimes\cdot\cdot\cdot\otimes p^{1}_{2}p^{2}_{2}...p^{m}_{2})^{\ast}\mathscr{I}(h).
\end{equation*}
\begin{proof}
We just prove the lemma with $m=3$, and the general case is the same. The calculation is nothing but using Lemma \ref{l1} and \ref{l2} repeatedly. Also we need the following two facts:

1. $(g\circ f)_{\ast}=g_{\ast}f_{\ast}$ for arbitrary morphisms $f$ and $g$;

2. $K_{X_{m}/Y}=(p^{m})^{\ast}_{1}K_{X/Y}\otimes (p^{m})^{\ast}_{2}K_{X_{m-1}/Y}$.

Then we complete the proof by carefully chasing the diagram.
\begin{equation*}
\begin{split}
 &(f_{3})_{\ast}(K_{X_{3}/Y}\otimes (q^{3})^{\ast}_{1}L\otimes (q^{3})^{\ast}_{2}(q^{2})^{\ast}_{1}L\otimes (q^{3})^{\ast}_{2}(q^{2})^{\ast}_{2}L)\\
=&f_{\ast}(q^{3}_{1})_{\ast}((q^{3}_{1})^{\ast}K_{X/Y}\otimes (q^{3}_{2})^{\ast}K_{X_{2}/Y}\otimes (q^{3}_{1})^{\ast}L\otimes (q^{3}_{2})^{\ast}(q^{2}_{1})^{\ast}L\otimes (q^{3}_{2})^{\ast}(q^{2}_{2})^{\ast}L)\\
=&f_{\ast}(K_{X/Y}\otimes L\otimes (q^{3}_{1})_{\ast}((q^{3}_{2})^{\ast}K_{X_{2}/Y}\otimes (q^{3}_{2})^{\ast}(q^{2}_{1})^{\ast}L\otimes (q^{3}_{2})^{\ast}(q^{2}_{2})^{\ast}L))\\
=&f_{\ast}(K_{X/Y}\otimes L\otimes f^{\ast}(f_{2})_{\ast}(K_{X_{2}/Y}\otimes (q^{2}_{1})^{\ast}L\otimes (q^{2}_{2})^{\ast}L)\\
=&f_{\ast}(K_{X/Y}\otimes L\otimes f^{\ast}(f_{\ast}(q^{2}_{1})_{\ast})((q^{2}_{1})^{\ast}K_{X/Y}\otimes (q^{2}_{2})^{\ast}K_{X/Y}\otimes (q^{2}_{1})^{\ast}L\otimes (q^{2}_{2})^{\ast}L)\\
=&f_{\ast}(K_{X/Y}\otimes L\otimes f^{\ast}f_{\ast}(K_{X/Y}\otimes L\otimes (q^{2}_{1})_{\ast}((q^{2}_{2})^{\ast}K_{X/Y}\otimes (q^{2}_{2})^{\ast}L))\\
=&f_{\ast}(K_{X/Y}\otimes L)\otimes f_{\ast}(K_{X/Y}\otimes L\otimes f^{\ast}f_{\ast}(K_{X/Y}\otimes L))\\
=&f_{\ast}(K_{X/Y}\otimes L)^{\otimes3}.
\end{split}
\end{equation*}
The last assertion comes from the subadditivity of the multiplier ideal sheaves proved in \cite{DEL00}.

In fact, since
\begin{equation*}
    h_{m}=(p^{m}_{1}\otimes p^{m-1}_{1}p^{m}_{2}\otimes p^{m-2}_{1}p^{m-1}_{2}p^{m}_{2}\otimes\cdot\cdot\cdot\otimes p^{1}_{2}p^{2}_{2}...p^{m}_{2})^{\ast}h,
\end{equation*}
we have
\begin{equation}\label{e5}
    \mathscr{I}(h_{m})\subset\mathscr{I}((p^{m}_{1})^{\ast}h)\otimes\mathscr{I}((p^{m-1}_{1}p^{m}_{2})^{\ast}h)\otimes\cdots\otimes\mathscr{I}((p^{1}_{2}p^{2}_{2}...p^{m}_{2})^{\ast}h)
\end{equation}
by the main result (Theorem 2.6) in \cite{DEL00}. One more application of Theorem 2.6 in \cite{DEL00} implies that
\begin{equation}\label{e6}
\begin{split}
\mathscr{I}((p^{m}_{1})^{\ast}h)&=(p^{m}_{1})^{\ast}\mathscr{I}(h),\\
\mathscr{I}((p^{m-1}_{1}p^{m}_{2})^{\ast}h)&=(p^{m-1}_{1}p^{m}_{2})^{\ast}\mathscr{I}(h),\\
\mathscr{I}((p^{m-2}_{1}p^{m-1}_{2}p^{m}_{2})^{\ast}h)&=(p^{m-2}_{1}p^{m-1}_{2}p^{m}_{2})^{\ast}\mathscr{I}(h),\\
    &\cdots\\
    \mathscr{I}((p^{1}_{2}p^{2}_{2}...p^{m}_{2})^{\ast}h)&=(p^{1}_{2}p^{2}_{2}...p^{m}_{2})^{\ast}\mathscr{I}(h).
\end{split}
\end{equation}

Indeed, on a local coordinate ball $U$ of $Y$, we have
\begin{equation*}
    X|_{f^{-1}(U)}=U\times X_{y}
\end{equation*}
and
\begin{equation*}
    X_{m}|_{f^{-1}_{m}(U)}=\underbrace{X_{y}\times\cdots\times X_{y}}_{m-1}\times(U\times X_{y}).
\end{equation*}
Thus using the notation of Theorem 2.6 in \cite{DEL00}, we apply its first statement to $X_{1}=\underbrace{X_{y}\times\cdots\times X_{y}}_{m-1}$ and $X_{2}=U\times X_{y}$ with $\phi_{1}=1$ and $\phi_{2}$ the weight function of $h$. Then we get
\begin{equation*}
\mathscr{I}((p^{m}_{1})^{\ast}h)=(p^{m}_{1})^{\ast}\mathscr{I}(h).
\end{equation*}
The other formulas are the same.

Combined with (\ref{e5}) and (\ref{e6}), the proof is finished.
\end{proof}
\end{lemma}

Furthermore, we need the following lemma concerning the behaviour of the singularity of a pseudo-effective metric after being pulled back through the fibre product.

\begin{lemma}\label{l4}
Let $f:X\rightarrow Y$ be a smooth fibration between two projective manifolds $X$ and $Y$. Let $(L,h)$ be a pseudo-effective line bundle on $X$. Moreover,
\begin{equation*}
    \psi_{l}(y)=(\int_{X_{y}}|u_{j}|^{2}_{h})^{l}\in L^{1}_{\mathrm{loc}}(Y)
\end{equation*}
for any $l\in\mathbb{Z}_{+}$ and $m$ local sections $u_{1},...,u_{m}$ of
\begin{equation*}
    \Gamma(U,f_{\ast}(K_{X/Y}\otimes L\otimes\mathscr{I}(h))).
\end{equation*}
Here $(U,(y_{1},...,y_{n}))$ is a coordinate neighborhood of $y\in Y$. Consider the $m$-fold fibre product $f_{m}:X_{m}=X\times_{Y}\cdot\cdot\cdot\times_{Y}X\rightarrow Y$. If we denote
\begin{equation*}
    L_{m}:=(p^{m}_{1}\otimes p^{m-1}_{1}p^{m}_{2}\otimes p^{m-2}_{1}p^{m-1}_{2}p^{m}_{2}\otimes\cdot\cdot\cdot\otimes p^{1}_{2}p^{2}_{2}...p^{m}_{2})^{\ast}(L),
\end{equation*}
and $h_{m}$ with weight function $\varphi_{m}$ being the metric induced by $h$, then
\begin{equation*}
    \int_{f^{-1}_{m}(U)}|u^{\otimes m}|^{2}e^{-\varphi_{m}}<+\infty.
\end{equation*}
Here $u^{\otimes m}$ is defined as
\begin{equation*}
    u^{\otimes m}=(p^{m}_{1})^{\ast}u_{1}\otimes(p^{m-1}_{1}p^{m}_{2})^{\ast}u_{2}\otimes \cdot\cdot\cdot\otimes(p^{1}_{2}p^{2}_{2}...p^{m}_{2})^{\ast}u_{m},
\end{equation*}
which is a section of
\begin{equation*}
   \Gamma(U,(f_{m})_{\ast}(K_{X_{m}/Y}\otimes L_{m}))=H^{0}(f^{-1}_{m}(U),K_{X_{m}/Y}\otimes L_{m})
\end{equation*}
by Lemma \ref{l3}. In particular, it means that
\begin{equation*}
    u^{\otimes m}\in\Gamma(U,(f_{m})_{\ast}(K_{X_{m}/Y}\otimes L_{m}\otimes\mathscr{I}(h_{m}))).
\end{equation*}
\begin{proof}
If we take the coordinate of $f^{-1}(U)$ to be $((y_{1},...,y_{n}),(x_{1},...,x_{l}))$, locally the weight function $\varphi$ of $h$ can be written as:
\begin{equation*}
   \varphi=\varphi((y_{1},...,y_{n}),(x_{1},...,x_{l})).
\end{equation*}
Since $f$ is smooth, we have $f^{-1}_{m}(U)=U\times X^{1}_{y}\times\cdot\cdot\cdot\times X^{m}_{y}$. Here we add the upper index $\{1,...,m\}$ to distinguish the fibres. Then we can take the local coordinate ball of $(X_{m})_{y}$ to be
\begin{equation*}
    (f^{-1}_{m}(U),((y_{1},...,y_{n}),(x^{1}_{1},...,x^{1}_{l}),...,(x^{m}_{1},...,x^{m}_{l}))),
\end{equation*}
and the weight function $\varphi_{m}$ of $h_{m}$ would be
\begin{equation*}
\begin{split}
\varphi_{m}&=(p^{m}_{1}+p^{m-1}_{1}p^{m}_{2}+p^{m-2}_{1}p^{m-1}_{2}p^{m}_{2}+...+ p^{1}_{2}p^{2}_{2}...p^{m}_{2})^{\ast}\varphi\\
           &=\sum_{j}\varphi((y_{1},...,y_{n}),(x^{j}_{1},...,x^{j}_{l})).
\end{split}
\end{equation*}
We claim that for any (local) section $u^{\otimes m}$ of
\begin{equation*}
   \Gamma(U,(f_{m})_{\ast}(K_{X_{m}/Y}\otimes L_{m}))=H^{0}(f^{-1}_{m}(U),K_{X_{m}/Y}\otimes L_{m})
\end{equation*}
defined above, the integral
\begin{equation*}
\begin{split}
\int_{U}\int_{X^{1}_{y}\times...\times X^{m}_{y}}|u^{\otimes m}|^{2}e^{-\varphi_{m}}
\end{split}
\end{equation*}
is finite. In fact, we have
\begin{equation*}
\begin{split}
\int_{U}\int_{X^{1}_{y}\times...\times X^{m}_{y}}|u^{\otimes m}|^{2}e^{-\varphi_{m}}
&:=\int_{U}\int_{X^{1}_{y}\times...\times X^{m}_{y}}|u^{\otimes m}|^{2}e^{-\sum_{j}\varphi((y_{1},...,y_{n}),(x^{j}_{1},...,x^{j}_{l}))}\\
&=\int_{U}\prod^{m}_{j}\int_{X^{j}_{y}}|u_{j}|^{2}e^{-\varphi((y_{1},...,y_{n}),(x^{j}_{1},...,x^{j}_{l}))}\\
&\leqslant\prod^{m}_{j}(\int_{U}(\int_{X^{j}_{y}}|u_{j}|^{2}e^{-\varphi((y_{1},...,y_{n}),(x^{j}_{1},...,x^{j}_{l}))})^{m})^{1/m}.
\end{split}
\end{equation*}
The last inequality is due to H\"{o}lder's inequality. Since
\begin{equation*}
    \int_{U}(\int_{X^{j}_{y}}|u_{j}|^{2}e^{-\varphi((y_{1},...,y_{n}),(x^{j}_{1},...,x^{j}_{l}))})^{m}=\int_{U}(\int_{X_{y}}|u_{j}|^{2}e^{-\varphi((y_{1},...,y_{n}),(x_{1},...,x_{l}))})^{m}
\end{equation*}
for every $j$, it is finite by assumption.

Then we conclude that
\begin{equation*}
    \int_{f^{-1}_{m}(U)}|u^{\otimes m}|^{2}e^{-\varphi_{m}}
\end{equation*}
is also finite. Indeed, let $Z:=\{y\in Y; \varphi|_{X_{y}}\equiv-\infty\}$, which is a Zariski closed subset of $Y$. For any $y_{0}\in Z$, we can take a family of tubular neighborhoods $\{f^{-1}_{m}(U_{\varepsilon})\}$ of $X_{y_{0}}$ such that $f^{-1}_{m}(U)\setminus f^{-1}_{m}(U_{\varepsilon})\subset f^{-1}_{m}(Y\setminus Z)$, and we have
\begin{equation*}
\begin{split}
    \int_{f^{-1}_{m}(U)\setminus f^{-1}_{m}(U_{\varepsilon})}|u^{\otimes m}|^{2}e^{-\varphi_{m}}&=\int_{(U\setminus U_{\varepsilon})\times X^{1}_{y}\times\cdot\cdot\cdot\times X^{m}_{y}}|u^{\otimes m}|^{2}e^{-\varphi_{m}}\\
    &\leqslant\int_{U}\int_{X^{1}_{y}\times\cdot\cdot\cdot\times X^{m}_{y}}|u^{\otimes m}|^{2}e^{-\varphi_{m}}
\end{split}
\end{equation*}
for every $\varepsilon>0$. Therefore we conclude that
\begin{equation*}
    \int_{f^{-1}_{m}(U)}|u^{\otimes m}|^{2}e^{-\varphi_{m}}=\lim_{\varepsilon\rightarrow0}\int_{f^{-1}_{m}(U)\setminus f^{-1}_{m}(U_{\varepsilon})}|u^{\otimes m}|^{2}e^{-\varphi_{m}}<+\infty.
\end{equation*}
\end{proof}
\end{lemma}

Now we turn to Theorem \ref{t2}.

\begin{proof}[Proof of Theorem \ref{t2}]
Consider the $m$-fold fibre product $f_{m}:X_{m}=X\times_{Y}\cdot\cdot\cdot\times_{Y}X\rightarrow Y$. If we denote
\begin{equation*}
    L_{m}:=(p^{m}_{1}\otimes p^{m-1}_{1}p^{m}_{2}\otimes p^{m-2}_{1}p^{m-1}_{2}p^{m}_{2}\otimes\cdot\cdot\cdot\otimes p^{1}_{2}p^{2}_{2}...p^{m}_{2})^{\ast}(L),
\end{equation*}
by Lemma \ref{l3} we have
\begin{equation*}
   (f_{\ast}(K_{X/Y}\otimes L))^{\otimes m}=(f_{m})_{\ast}(K_{X_{m}/Y}\otimes L_{m}).
\end{equation*}
In order to apply Corollary \ref{c1}, we should analysis the singularity of the pseudo-effective metric of $L_{m}$. Indeed, the metric $h$ of $L$ induces a natural metric
\begin{equation*}
   h_{m}=(p^{m}_{1}\otimes p^{m-1}_{1}p^{m}_{2}\otimes p^{m-2}_{1}p^{m-1}_{2}p^{m}_{2}\otimes...\otimes p^{1}_{2}p^{2}_{2}...p^{m}_{2})^{\ast}(h)
\end{equation*}
of $L_{m}$ with positive curvature current. Lemma \ref{l3} implies that
\begin{equation*}
    \mathscr{I}(h_{m})\subset(p^{m}_{1}\otimes p^{m-1}_{1}p^{m}_{2}\otimes p^{m-2}_{1}p^{m-1}_{2}p^{m}_{2}\otimes\cdot\cdot\cdot\otimes p^{1}_{2}p^{2}_{2}...p^{m}_{2})^{\ast}\mathscr{I}(h),
\end{equation*}
so we have
\begin{equation*}
\begin{split}
    &(f_{m})_{\ast}(K_{X_{m}/Y}\otimes L_{m}\otimes\mathscr{I}(h_{m}))\\
    \subset&(f_{m})_{\ast}(K_{X_{m}/Y}\otimes L_{m}\otimes(p^{m}_{1}\otimes p^{m-1}_{1}p^{m}_{2}\otimes\cdot\cdot\cdot\otimes p^{1}_{2}p^{2}_{2}...p^{m}_{2})^{\ast}\mathscr{I}(h))\\
    =&f_{\ast}(K_{X/Y}\otimes L\otimes\mathscr{I}(h))^{\otimes m}.
\end{split}
\end{equation*}
While Lemma \ref{l4} says that the opposite direction holds under the assumption that
\begin{equation*}
    \psi_{l}(y)=(\int_{X_{y}}|u|^{2}_{h})^{l}\in L^{1}_{\mathrm{loc}}(Y)
\end{equation*}
for any $l\in\mathbb{Z}_{+}$ and local section $u$ of $f_{\ast}(K_{X/Y}\otimes L\otimes\mathscr{I}(h))$. So we actually have
\begin{equation*}
    \mathcal{E}^{\otimes m}=f_{\ast}(K_{X/Y}\otimes L\otimes\mathscr{I}(h))^{\otimes m}=(f_{m})_{\ast}(K_{X_{m}/Y}\otimes L_{m}\otimes\mathscr{I}(h_{m})).
\end{equation*}

Then we fix a very ample line bundle $H$ over $Y$ and let $A=K_{Y}\otimes H^{\otimes(n+1)}$ with $n=\dim Y$. Applying Corollary \ref{c1} to the fibration $f_{m}$ and the direct image $(f_{m})_{\ast}(K_{X_{m}/Y}\otimes L_{m}\otimes\mathscr{I}(h_{m}))$, we deduce that the sheaf $\mathcal{E}^{\otimes m}\otimes A$ is generated by its global sections.

Therefore the vector bundle $S^{m}\mathcal{E}\otimes A$, being a quotient of $\mathcal{E}^{\otimes m}\otimes A$, is globally generated, too. Consider $\pi:\mathbb{P}(\mathcal{E}^{\ast})\rightarrow Y$. Note that we have a surjective morphism
\begin{equation*}
    \pi^{\ast}\pi_{\ast}\mathcal{O}_{\mathcal{E}}(m)\cong\pi^{\ast}(S^{m}(\mathcal{E}))\rightarrow\mathcal{O}_{\mathcal{E}}(m),
\end{equation*}
and we thus deduce that $\mathcal{O}_{\mathcal{E}}(m)\otimes\pi^{\ast}A$ is globally generated, hence nef, for every $m\geqslant1$. This implies that $\mathcal{O}_{\mathcal{E}}(1)$ is nef, that is, $\mathcal{E}$ is nef.
\end{proof}

\section{Two examples}
In this section, we will present two interesting examples.

Firstly, we do some general computation. Since the crucial thing is the local integrability of the function
\begin{equation*}
    \psi_{l}(y)=(\int_{X_{y}}|u|^{2}e^{-\varphi})^{l}=e^{l\log\int_{X_{y}}|u|^{2}e^{-\varphi}}
\end{equation*}
on $Y$, we will calculate the $(1,1)$-current
\begin{equation*}
    -dd^{c}_{y}\log\int_{X_{y}}|u|^{2}e^{-\varphi}.
\end{equation*}
Here $dd^{c}_{y}$ means to take the derivative with respect to $y$. We focus on a local coordinate ball $U$ of $Y$. Take an arbitrary $u$ in
\begin{equation*}
   \Gamma(U,f_{\ast}(K_{X/Y}\otimes L\otimes\mathscr{I}(\varphi))),
\end{equation*}
then $-\log\int_{X_{y}}|u|^{2}e^{-\varphi}$ is a function on (a subset $U$ of) $Y$ not identically equaling to $\infty$.

We begin with a general setting. Consider the trivial vector bundle $E:=f_{\ast}(K_{X/Y}\otimes L)$ over $U$ with fibre $H^{0}(X_{y},K_{X_{y}}\otimes L)$. $\mathrm{rank}E=h^{0}(X_{y},K_{X_{y}}\otimes L)=r$. Now we do some reduction. Firstly, $\varphi$ can be assumed to be smooth by approximation. Then
\begin{equation*}
    \|u\|^{2}_{y}=\int_{X_{y}}|u|^{2}e^{-\varphi}.
\end{equation*}
is the $L^{2}$-metric of $E$. We furthermore assume that $\dim Y=1$ to ease the notation. As is proved in \cite{Ber09}, the curvature associated with this $L^{2}$-metric can be written as
\begin{equation}\label{e7}
    \langle\Theta^{E}u,u\rangle=(\int_{X_{y}}dd^{c}\varphi\wedge u\wedge\bar{u}e^{-\varphi})-\|P_{\perp}\mu\|^{2}-c_{n}\int_{X_{y}}\eta\wedge\bar{\eta}e^{-\varphi}.
\end{equation}
Here $\mu$ is given by $\partial^{\varphi}u=dy\wedge\mu$ and $\eta$ by $\bar{\partial}u=dy\wedge\eta$. $P_{\perp}$  is the orthogonal projection on the orthogonal complement of holomorphic forms. Moreover, $\langle\Theta^{E}u,u\rangle$ is positive since $\varphi$ is plurisubharmonic.

Now let $Z$ be the total space of $E$ and $\pi:Z\rightarrow U$ be the projection. Notice that
\begin{equation*}
    \log(\int_{X_{y}}|u|^{2}e^{-\varphi})=\psi(u,y)
\end{equation*}
can be seen as a function on $Z$, which is log-homogeneous with respect to $u$. In other word, we see $u$ as an independent variable instead of a function of $y$. Since $\|u\|^{2}$ is a Hermitian metric, the Levi form of $\psi$ along $Z_{y}$ is strictly positive. (Or we can say that the Finsler metric $\psi$ induced by $\|u\|^{2}$ is strongly pseudoconvex.)

We now expand $dd^{c}_{u,y}\psi$ on the $\pi^{-1}(U)$ (not on $U$) as
\begin{equation*}
\begin{split}
    &dd^{c}_{u,y}(\log\int_{X_{y}}|u|^{2}e^{-\varphi})=\\
    &i(g_{1\bar{1}}dy\wedge d\bar{y}+g_{1\bar{\beta}}dy\wedge d\bar{u}^{\beta}+g_{\alpha\bar{1}}du^{\alpha}\wedge d\bar{y}+g_{\alpha\bar{\beta}}du^{\alpha}\wedge d\bar{u}^{\beta}).
\end{split}
\end{equation*}
Here $(u^{\alpha})$ is the vertical coordinate of $\pi:Z\rightarrow U$ and $y$ the horizontal one. $dd^{c}_{u,y}$ means to take the derivative with respect to $(u^{\alpha})$ and $y$. Since $\psi$ is strongly pseudoconvex, the matrix $(g_{\alpha\bar{\beta}})$ is invertible. So we can define the coformal basis by
\begin{equation*}
   \{\delta u^{\alpha}=du^{\alpha}+g^{\Bar{\beta}\alpha}g_{1\Bar{\beta}}dy,dy\}.
\end{equation*}

It is proved by Kobayashi \cite{Kob75} that on $\pi^{-1}(U)$ (not on $U$), we have
\begin{equation}\label{e8}
    dd^{c}_{u,y}\psi=dd^{c}_{u,y}(\log\int_{X_{y}}|u|^{2}e^{-\varphi})=-\Psi+\omega_{V},
\end{equation}
where $\Psi=ig_{1\bar{1}}dy\wedge d\bar{y}=i\Theta^{E}_{\alpha\Bar{\beta}1\Bar{1}}\frac{u^{\alpha}\Bar{u}^{\beta}}{\int_{X_{y}}|u|^{2}e^{-\varphi}}dy\wedge d\Bar{y}$ and
\begin{equation*}
    \omega_{V}=i\frac{\partial^{2}(\log\int_{X_{y}}|u|^{2}e^{-\varphi})}{\partial u^{\alpha}\partial\Bar{u}^{\beta}}\delta u^{\alpha}\wedge\delta\Bar{u}^{\beta}.
\end{equation*}
Obviously, $\Psi$ and $\omega_{V}|_{Z_{y}}$ are both positive.

We remark here that if (\ref{e8}) is pulled back to $U$ through a holomorphic section $u:U\rightarrow\pi^{-1}(U)$, we get
\begin{equation}\label{e9}
\begin{split}
    -dd^{c}_{y}(\log\int_{X_{y}}|u|^{2}e^{-\varphi})&=u^{\ast}\Psi-u^{\ast}\omega_{V}.
\end{split}
\end{equation}
It is just the standard formula
\begin{equation*}
\begin{split}
    -\partial\bar{\partial}(\log\|u\|^{2})&=-(\frac{\partial\bar{\partial}\|u\|^{2}}{\|u\|^{2}}-i\frac{\langle D^{1,0}u,u\rangle\langle u,D^{1,0}u\rangle}{\|u\|^{4}})\\
    &=\frac{\langle\Theta^{E}u,u\rangle-\langle D^{1,0}u,D^{1,0}u\rangle}{\|u\|^{2}}+\frac{\langle D^{1,0}u,u\rangle\langle u,D^{1,0}u\rangle}{\|u\|^{4}},
\end{split}
\end{equation*}
where $D^{1,0}$ is the $(1,0)$-part of the Chern connection associated to $\|u\|^{2}$.

Now we focus on a special case. Assume that $-K_{X/Y}\geqslant0$ and take $L=-K_{X/Y}$. In this situation, $H^{0}(X_{y},K_{X_{y}}\otimes L)$ only has a trivial element $u=1$, so $E=f_{\ast}(K_{X/Y}\otimes L)$ is actually a line bundle. Moreover, $-dd^{c}_{y}(\log\int_{X_{y}}e^{-\varphi}dV_{y})$, which is a closed positive $(1,1)$-form on $U$, is exactly the curvature associated with the $L^{2}$-metric of $E$. Here we use the classic equivalence between the curvature $\Theta^{E}\in A^{1,1}(\mathrm{End}(E))$ and the first Chern form $c_{1}(E)\in H^{1,1}(Y)$ for a line bundle $E$. Therefore (\ref{e9}) can be simplified as
\begin{equation*}
\begin{split}
    -dd^{c}_{y}(\log\int_{X_{y}}e^{-\varphi}dV_{y})&=\frac{\int_{X_{y}}dd^{c}_{X}\varphi\wedge e^{-\varphi}dV_{y}}{\int_{X_{y}}e^{-\varphi}dV_{y}}.
\end{split}
\end{equation*}
Notice that $\mu$ and $\eta$ in (\ref{e7}) equal to zero here. The singular case comes from the standard approximation. In particular, $-\log\int_{X_{y}}e^{-\varphi}$ is a subharmonic function on $U$ provided that $\varphi$ is a plurisubharmonic function. This result has already been proved in \cite{Ber15} as a complex counterpart of the functional version of the Brunn--Monkowski inequality.

Finally, if the Lelong number of $-\log\int_{X_{y}}e^{-\varphi}$, which is well-defined, equals to zero, then $(\int_{X_{y}}e^{-\varphi})^{l}$ must be locally integrable for all $l$ by \cite{Sko72a}. We summarize the discussion above to such a definition:

\begin{definition}[Lelong number along fibre]
Let $f:X\rightarrow Y$ be a fibration between two projective manifolds $X$ and $Y$. $\varphi$ is the weight function of a pseudo-effective metric of $-K_{X/Y}$ over $X$. Let $y\in Y$ be an arbitrary point. Then the Lelong number along fibre $X_{y}$ of $\varphi$ is defined to be
\begin{equation*}
  \nu(\varphi, X_{y}):=\nu(-\log\int_{X_{y}}e^{-\varphi}, y).
\end{equation*}
\end{definition}

We list a few basic properties.
\begin{proposition}\label{p1}
Let $f:X\rightarrow Y$ be the fibration considered above. $\dim Y=n$.
\begin{enumerate}
  \item For every plurisubharmonic function $\varphi$, the Lelong number along fibre $\nu(\varphi,X_{y})$ always exists.
  \item $\nu(\varphi,X_{y})=\sup\{\gamma|\int_{X_{y}}e^{-\varphi}\geqslant\frac{C}{|z-y|^{\gamma}}\textrm{ for some constant C at }y\}$.
  \item If $\nu(\varphi,X_{y})<1$, then $\int_{X_{y}}e^{-\varphi}$ is integrable in a neighborhood of $y$.
  \item If $\nu(\varphi,X_{y})\geqslant n+s$ for some integer $s\geqslant0$, then $\int_{X_{y}}e^{-\varphi}\geqslant C|z-y|^{-2n-2s}$ in a neighborhood of $y$.
\end{enumerate}
\begin{proof}
(1) has been discussed before.

(2) is a direct consequence of \cite{Lel57}.

(3) If $U$ is a neighborhood of $y$, we have
\begin{equation*}
\begin{split}
    \int_{U}\int_{X_{y}}e^{-\varphi}&=\int_{U}e^{-(-\log\int_{X_{y}}e^{-\varphi})}.
\end{split}
\end{equation*}
Since the Lelong number of $-\log\int_{X_{y}}e^{-\varphi}$ is less than one, it is a quick consequence of \cite{Sko72a} that $\int_{U}\int_{X_{y}}e^{-\varphi}<\infty$.

(4) is also a quick consequence of \cite{Sko72a}.
\end{proof}
\end{proposition}

The first example is as follows.

\begin{example}
Let $f:X\rightarrow Y$ be a smooth fibration between two projective manifolds $X$ and $Y$. Assume that $(-K_{X/Y},\varphi)$ is pseudo-effective with $\nu(\varphi,X_{y})=0$ for all $y\in Y$. Then $f_{\ast}(\mathscr{I}(\varphi))$ is nef (as a torsion-free coherent sheaf) by Theorem \ref{t2}.

In fact, for any local coordinate ball $U$,
\begin{equation*}
    \int_{U}\int_{X_{y}}e^{-\varphi}<\infty
\end{equation*}
since $\nu(\varphi,X_{y})=0$. So we actually have $f_{\ast}(\mathscr{I}(\varphi))=f_{\ast}(\mathcal{O}_{X})=\mathcal{O}_{Y}$.
\end{example}

Another example comes from the Hermitian--Einstein theory.
\begin{example}\label{ex1}
Let $E$ be a stable vector bundle of rank $r$ over a compact Riemann surface $Y$ of genus $\geqslant2$. Then it is proved in \cite{NaS65} that there exists a coordinate chart $\{U_{i}\}$ of $Y$ such that the transition matrices can be written in the form $g_{ij}=f_{ij}U_{ij}$, where $f_{ij}$ is a scalar function and $U_{ij}$ is a unitary matrix on $U_{i}\cap U_{j}$.

Assume that $(\det E,\{h_{i}\})$ is pseudo-effective. Then $\{h^{1/r}_{i}I_{r}\}$ defines a (singular) Hermitian metric $h$ on $E$. Here $I_{r}$ is the identity matrix of rank $r$. If we denote the weight function of $\{h_{i}\}$ by $\phi$, the associated curvature of $h$ is given by $\frac{1}{r}h^{1/r}_{i}I_{r}dd^{c}\phi$, which is semi-positive in the sense of Griffiths. Take $\psi$ to be the weight function of the corresponded Finsler metric of $h$ on $\mathcal{O}_{E}(1)$. Let
\begin{equation*}
\begin{split}
    j:E^{\ast}&\rightarrow\mathbb{P}(E^{\ast})\\
    (y,(u^{\alpha}))&\mapsto(y,[u^{\alpha}])
\end{split}
\end{equation*}
be the natural morphism and $(y,(\frac{u^{\alpha}}{u^{A}}),\xi)$ be the local coordinate of $\mathcal{O}_{E}(1)$. Here $(\frac{u^{\alpha}}{u^{A}})$ refers to the local coordinate along fibre on the open set $\{u^{A}\neq0\}$. We furthermore let $\pi:X:=\mathbb{P}(E^{\ast})\rightarrow Y$ be the projection. The formula (\ref{e8}) then gives that
\begin{equation*}
    \begin{split}
        dd^{c}\psi&=dd^{c}\log h^{\ast}+dd^{c}\log|\xi|^{2}\\
                  &=j_{\ast}(-\Psi+\omega_{V})+[\{\xi=0\}]\\
                  &=j_{\ast}(-i\Theta^{E^{\ast}}_{\alpha\bar{\beta}i\bar{j}}\frac{u^{\alpha}\bar{u}^{\beta}}{h^{\ast}}dy^{i}\wedge d\bar{y}^{j}+i\frac{\partial^{2}(\log h^{\ast})}{\partial u^{\alpha}\partial\bar{u}^{\beta}}\delta u^{\alpha}\wedge\delta\bar{u}^{\beta})+[\{\xi=0\}]\\
                  &=\frac{1}{r}\pi^{\ast}(h^{-1/r}_{i}dd^{c}\phi\frac{\sum_{\alpha}|u^{\alpha}/u^{A}|^{2}}{h^{-1/r}_{i}\sum_{\alpha}|u^{\alpha}/u^{A}|^{2}})+i\frac{d\frac{u^{\alpha}}{u^{A}}\wedge d\overline{\frac{u^{\alpha}}{u^{A}}}}{(1+\sum|\frac{u^{\alpha}}{u^{A}}|^{2})^{2}}\\
                  &=\frac{1}{r}dd^{c}\pi^{\ast}\phi+i\frac{d\frac{u^{\alpha}}{u^{A}}\wedge d\overline{\frac{u^{\alpha}}{u^{A}}}}{(1+\sum|\frac{u^{\alpha}}{u^{A}}|^{2})^{2}}.
    \end{split}
\end{equation*}
$[\{\xi=0\}]$ is the current of integration of the zero section. In the forth equality, we use the fact that $\xi$ has the same transition function as $\frac{1}{u^{A}}$. Obviously $dd^{c}\psi$ is a closed positive $(1,1)$-current. It means that $(\mathcal{O}_{E}(1),\psi)$ is pseudo-effective. We endow the line bundle $\mathcal{O}_{E}(r+1)\otimes\pi^{\ast}\det E^{\ast}$ with the metric $\varphi=(r+1)\psi-\pi^{\ast}\phi$. Since
\begin{equation*}
\begin{split}
    dd^{c}\varphi&=\frac{r+1}{r}dd^{c}\pi^{\ast}\phi-dd^{c}\pi^{\ast}\phi+i(r+1)\frac{d\frac              {u^{\alpha}}{u^{A}}\wedge       d\overline{\frac{u^{\alpha}}{u^{A}}}}{(1+\sum|\frac{u^{\alpha}}{u^{A}}|^{2})^{2}}\\
                 &=\frac{1}{r}dd^{c}\pi^{\ast}\phi+i(r+1)\frac{d\frac{u^{\alpha}}{u^{A}}\wedge d\overline{\frac{u^{\alpha}}{u^{A}}}}{(1+\sum|\frac{u^{\alpha}}{u^{A}}|^{2})^{2}}\geqslant0,
\end{split}
\end{equation*}
the line bundle $(\mathcal{O}_{E}(r+1)\otimes\pi^{\ast}\det E^{\ast},\varphi)$ is also pseudo-effective. Moreover, if we treat an $\mathcal{O}_{E}(r+1)\otimes\pi^{\ast}\det E^{\ast}$-valued $(r-1,0)$-form $\eta$ as a (local) section $\xi$ of $\mathcal{O}_{E}(1)$ via the following isomorphism (which is even an isometry here)
\begin{equation*}
   \mathcal{O}_{E}(1)=K_{X/Y}\otimes\mathcal{O}_{E}(r+1)\otimes\pi^{\ast}\det E^{\ast},
\end{equation*}
we have
\begin{equation*}
    \int_{X_{y}}|\eta|^{2}e^{-\varphi}=\int_{X_{y}}|\xi|^{2}e^{-\frac{1}{r}\pi^{\ast}\phi}=h^{1/r}_{i}\sum|u^{\alpha}|^{2}.
\end{equation*}
The last equality is due to the isomorphism
\begin{equation*}
    \begin{split}
        \pi_{\ast}\mathcal{O}_{E}(1)&\rightarrow E\\
        \xi &\mapsto (u^{\alpha}).
    \end{split}
\end{equation*}
It means that $\int_{X_{y}}|\eta|^{2}e^{-\varphi}$ is nothing but the norm of the section $(u^{\alpha})$ of $E$ corresponded to $\eta$ with respect to the singular Hermitian metric $h$ defined at the beginning. In other word, for any positive integer $l$ and local section $\eta$ of
\begin{equation*}
    \pi_{\ast}(K_{X/Y}\otimes\mathcal{O}_{E}(r+1)\otimes\pi^{\ast}\det E^{\ast}),
\end{equation*}
$(\int_{X_{y}}|\eta|^{2}e^{-\varphi})^{l}\in L^{1}_{\mathrm{loc}}(Y)$ if and only if $h^{l/r}_{i}\in L^{1}_{\mathrm{loc}}(Y)$.

Apply Theorem \ref{t2} with the fibration $\pi:X=\mathbb{P}(E^{\ast})\rightarrow Y$ and the pseudo-effective line bundle $L=\mathcal{O}_{E}(r+1)\otimes\pi^{\ast}\det E^{\ast}$ over $X$, we have the following conclusion.

\textbf{Conclusion}: Let $E$ be a stable vector bundle of rank $r$ over a compact Riemann surface $Y$ of genus $\geqslant2$. Assume that $(\det E,\{h_{i}\})$ is pseudo-effective. If $\mathscr{I}(h^{l}_{i})=\mathcal{O}_{Y}$ for all positive integer $l$, then $E$ is nef.
\end{example}

Based on this example, we can prove Theorem \ref{t3}.
\begin{proof}[Proof of Theorem \ref{t3}]
The case that $g(Y)\geqslant2$ has been shown in Example \ref{ex1}.

$g(Y)=0$ is simple. In fact, $Y=\mathbb{P}^{1}$ at this time. The only stable bundles are $\mathcal{O}(a)$ for all the integers $a$. So if $\det\mathcal{O}(a)=\mathcal{O}(a)$ is pseudo-effective, $a\geqslant0$. So $\mathcal{O}(a)$ is nef.

$g(Y)=1$ is also not complicated. In this situation, we have
\begin{theorem}[Atiyah--Tu,\cite{Ati57,Tu93}]
Let $Y$ be an elliptic curve. Denote the moduli space of stable vector bundles of rank $r$ and degree $d$ by $\mathcal{M}_{Y}(r,d)$. Then the determine map
\begin{equation*}
    \det:\mathcal{M}_{Y}(r,d)\rightarrow\mathrm{Pic}^{d}_{Y}
\end{equation*}
is an isomorphism of complex analytic manifolds of dimension 1. Here $\mathrm{Pic}^{d}_{Y}$ refers to the set of all the equivalence classes of the line bundle with degree $d$ in $\mathrm{Pic}(Y)$.
\end{theorem}
The procedure of this bijective map only involves a line bundle $A$ of degree $r$ and the trivial vector bundle $I_{l}$ of rank $l$. First, let $d=\mathrm{deg}E$. Since $E\mapsto A\otimes E$ is a bijiection from the set of vector bundles with rank $r$ and degree $d$ to the set of vector bundles with rank $r$ and degree $d+r$, one can always assume that $\deg E<r$. Then $E$ can be written as
\begin{equation}\label{e10}
    0\rightarrow I_{d}\rightarrow E\rightarrow E^{\prime}\rightarrow0
\end{equation}
by \cite{Ati57}. So $E^{\prime}$ is a vector bundle with rank $r-d$ and degree $d$. Moreover, $E^{\prime}$ is nef if and only if $E$ is by \cite{DPS94}. Once more, we can reduce it to the case that $\deg E^{\prime}<r-d$ and get an exact sequence similar with (\ref{e10}). Repeating this program, finally we get a vector bundle $E^{\prime\prime}$ with rank $1$ and degree $0$. (Remember that the highest common factor $(r,d)=1$ since $E$ is stable.) Since $E^{\prime\prime}$ is nef, so will be $E$. The proof is complete.
\end{proof}

\section{Appendix}
In this section we will prove a general version of the nefness of the direct images via the Ohsawa--Takegoshi theorem in order to illustrate the difference between our method and the analytic method used in \cite{Ber08,PT14}. This result is more or less included in those papers.

\begin{theorem}\label{t7}
Let $f:X\rightarrow Y$ be a fibration between two projective manifolds $X$ and $Y$. $(L,h)$ is a pseudo-effective line bundle on $X$ with $\mathscr{I}(h)=\mathcal{O}_{X}$. Assume that $\mathcal{E}=f_{\ast}(K_{X/Y}\otimes L)$ is locally free. Then it is pseudo-effective.
\end{theorem}
In fact, $\mathcal{E}$ is proved to be positively curved even without the assumption that $\mathscr{I}(h)=\mathcal{O}_{X}$ in \cite{Ber08}. There is no obvious relationship between being positively curved and nef. But we know that when $\textrm{rank}\mathcal{E}=1$, being positively curved is equivalent to being pseudo-effective, which is weaker than being nef. Therefore we believe that there is something new in our result.

Technically, the proof of Theorem \ref{t7} follows the same way as in the proof of Theorem \ref{t2}, i.e. to prove that $f_{\ast}(K_{X/Y}\otimes L)^{\otimes m}\otimes A$ is generically globally generated for an auxiliary line bundle $A$ and all $m>0$. The difference is that we will use the Ohsawa--Takegoshi theorem instead of Theorem \ref{t1} to achieve it.
\begin{proof}[Proof of Theorem \ref{t7}]
For a vector bundle, being pseudo-effective is equivalent to being weakly positive in the sense of Viehweg. So it is sufficient to find a Zariski open subset $Y^{\prime}\subset Y$ such that
\begin{equation*}
  \mathcal{E}^{\otimes m}\otimes A
\end{equation*}
is globally generated on $Y^{\prime}$ for a fixed line bundle $A$ and all $m>0$. Since $\mathscr{I}(h)=\mathcal{O}_{X}$, there is a Zariski open subset $Z^{\prime}\subset Y$ such that $\mathscr{I}(h|_{X_{y}})=\mathcal{O}_{X_{y}}$ for all $y\in Z^{\prime}$ by Fubini's Theorem. Let $Z\subset Y$ be the set of all regular value of $f$. Then we take $Y^{\prime}=Z\cap Z^{\prime}$ and $A=K_{Y}\otimes H^{\otimes(\dim Y+1)}$, where $H$ is a very ample line bundle over $Y$. Following the same idea of Theorem \ref{t2}, we consider the $m$-fold fibre product $f_{m}:X_{m}=X\times_{Y}\cdot\cdot\cdot\times_{Y}X\rightarrow Y$. If we denote
\begin{equation*}
    L_{m}:=(p^{m}_{1}\otimes p^{m-1}_{1}p^{m}_{2}\otimes p^{m-2}_{1}p^{m-1}_{2}p^{m}_{2}\otimes\cdot\cdot\cdot\otimes p^{1}_{2}p^{2}_{2}...p^{m}_{2})^{\ast}(L),
\end{equation*}
by Lemma \ref{l3} we have
\begin{equation*}
    \mathcal{E}^{\otimes m}\otimes A=(f_{m})_{\ast}(K_{X_{m}/Y}\otimes L_{m})\otimes A.
\end{equation*}
So we only need to prove the global generation of $(f_{m})_{\ast}(K_{X_{m}/Y}\otimes L_{m})\otimes A$ on $Y^{\prime}$.

Note that the bundle $L_{m}$ with the metric
\begin{equation*}
    h_{m}=(p^{m}_{1}\otimes p^{m-1}_{1}p^{m}_{2}\otimes p^{m-2}_{1}p^{m-1}_{2}p^{m}_{2}\otimes\cdot\cdot\cdot\otimes p^{1}_{2}p^{2}_{2}...p^{m}_{2})^{\ast}(h)
\end{equation*}
is pseudo-effective. Moreover, because the fibre $(X_{m})_{y}$ is isomorphic to the $m$-fold product $X_{y}\times\cdot\cdot\cdot\times X_{y}$ for $y\in Y^{\prime}$, we have
\begin{equation*}
   \mathscr{I}(h_{m}|_{(X_{m})_{y}})=(p^{m}_{1}\otimes p^{m-1}_{1}p^{m}_{2}\otimes p^{m-2}_{1}p^{m-1}_{2}p^{m}_{2}\otimes\cdot\cdot\cdot\otimes p^{1}_{2}p^{2}_{2}...p^{m}_{2})^{\ast}\mathscr{I}(h|_{X_{y}})
\end{equation*}
by the subadditivity theorem \cite{DEL00}. Since $\mathscr{I}(h|_{X_{y}})=\mathcal{O}_{X_{y}}$, combining with the Ohsawa--Takegoshi theorem we have that
\begin{equation*}
   \mathcal{O}_{(X_{m})_{y}}=\mathscr{I}(h_{m}|_{(X_{m})_{y}})\subset\mathscr{I}(h_{m})|_{(X_{m})_{y}}.
\end{equation*}
Here we use the local version of the Ohsawa--Takegoshi theorem.
\begin{theorem}(local version)\label{t8}
Let $f:X\rightarrow\Delta$ be a smooth fibration between projective manifold $X$ and the disc $\Delta$. $(L,h)$ is a (singular) Hermitian line bundle with semi-positive curvature current $i\Theta_{L,h}$ on $X$. Let $\omega$ be a global K\"{a}hler metric on $X$, and $dV_{X},dV_{X_{0}}$ the respective induced volume forms on $X$ and $X_{0}$. Here $X_{0}$ is the central fibre. Assume that $h|_{X_{0}}$ is well-defined. Then any holomorphic section $u$ of $K_{X_{0}}\otimes L|_{X_{0}}\otimes\mathscr{I}(\varphi|_{X_{0}})$ extends into a section $\tilde{u}$ over $X$ satisfying an $L^{2}$-estimate
\begin{equation*}
   \int_{X}|\tilde{u}|^{2}_{\omega,h}dV_{\omega}\leqslant C_{0}\int_{X_{0}}|u|^{2}_{\omega,h}dV_{X_{0}},
\end{equation*}
where $C_{0}$ is some universal constant.
\end{theorem}
Therefore $\mathscr{I}(h_{m})|_{(X_{m})_{y}}=\mathcal{O}_{(X_{m})_{y}}$ for all $y\in Y^{\prime}$.

On the other hand, we have
\begin{equation*}
   H^{0}(Y,(f_{m})_{\ast}(K_{X_{m}/Y}\otimes L_{m})\otimes A)=H^{0}(X_{m},K_{X_{m}/Y}\otimes L_{m}\otimes f^{\ast}_{m}A)
\end{equation*}
and the fibre of $(f_{m})_{\ast}(K_{X_{m}/Y}\otimes L_{m})\otimes A$ equals to (remember that $A_{y}$ is just $\mathbb{C}$)
\begin{equation*}
  H^{0}((X_{m})_{y}, K_{(X_{m})_{y}}\otimes L_{m}|_{(X_{m})_{y}}),
\end{equation*}
we only need to show that the restriction morphism
\begin{equation*}
H^{0}(X_{m},K_{X_{m}}\otimes L_{m}\otimes f^{\ast}_{m}H^{\otimes(\dim Y+1)})\rightarrow H^{0}((X_{m})_{y}, K_{(X_{m})_{y}}\otimes L_{m}|_{(X_{m})_{y}})
\end{equation*}
is surjective (remember that we set $A=K_{Y}\otimes H^{\otimes(\dim Y+1)}$ at the beginning). We do the induction on the dimension of $Y$.

If $\dim Y=1$, a general element $D$ of the linear system $|f^{\ast}_{m}H|$ is a disjoint union of smooth fibres. Hence if $(X_{m})_{y}=f^{-1}_{m}(y)$ for some $y\in f(D)$, we have the surjective restriction morphism
\begin{equation*}
   H^{0}(D,K_{D}\otimes L_{m}|_{D})\rightarrow H^{0}((X_{m})_{y}, K_{(X_{m})_{y}}\otimes L_{m}|_{(X_{m})_{y}}).
\end{equation*}
Fix a smooth metric $h^{\prime}$ with positive curvature on $H$. Since $\Theta_{h_{m}}\geqslant0$, we have
\begin{equation*}
   \Theta_{h_{m}\otimes f^{\ast}_{m}(h^{\prime})^{\otimes2}}(L_{m}\otimes f^{\ast}_{m}H^{\otimes2})\geqslant\Theta_{f^{\ast}_{m}(h^{\prime})^{\otimes2}}(f^{\ast}_{m}H^{\otimes2})\geqslant\Theta_{f^{\ast}_{m}(h^{\prime})}(f^{\ast}_{m}H).
\end{equation*}
Thus the line bundle $L_{m}\otimes f^{\ast}_{m}H^{\otimes2}$ endowed with the metric $h_{m}\otimes f^{\ast}_{m}(h^{\prime})^{\otimes2}$ satisfies the conditions of the following Ohsawa--Takegoshi theorem:
\begin{theorem}(global version)\label{t9}
Let $X$ be an $n$-dimensional projective or Stein manifold, $Z\subset X$ is the zero locus of some section $s\in H^{0}(X,E)$ for a holomorphic vector bundle $E\rightarrow X$; suppose that $Z$ is smooth and the codimension $r=\mathrm{rank}(E)$. Let $(F,h)$ be a line bundle, endowed with (singular) Hermitian metric $h$ with

(1) $\Theta_{h}(F)+i\partial\bar{\partial}\log|s|^{2}\geqslant0$;

(2) $\Theta_{h}(F)+i\partial\bar{\partial}\log|s|^{2}\geqslant(<\Theta(E)s,s>)/(\alpha|s|^{2})$ for some $\alpha\geqslant1$;

(3) $|s|^{2}\leqslant\exp(-\alpha)$, and $h|_{Z}$ is well-defined.

Then all the sections $u\in H^{0}(Z,(K_{X}+F|_{Z})\otimes\mathscr{I}(h|_{Z}))$ extend to $\tilde{u}$ on the whole $X$, and satisfy that
\begin{equation*}
   \int_{X}\frac{\tilde{u}\wedge\bar{\tilde{u}}\exp(-\varphi_{F})}{|s|^{2r}(-\log|s|^{2})}\leqslant C_{\alpha}\int_{Z}\frac{|u|^{2}\exp(-\varphi_{F})dV_{Z}}{|\bigwedge^{r}(ds)|^{2}}.
\end{equation*}
Here $h=e^{-\varphi_{F}}$.
\end{theorem}
Therefore any section in
\begin{equation*}
    H^{0}(D,K_{D}\otimes L_{m}|_{D}\otimes\mathscr{I}(h_{m}|_{D}))
\end{equation*}
extends to a section in $H^{0}(X_{m},K_{X_{m}}\otimes L_{m}\otimes\mathscr{I}(h_{m}))$. Since for any $y\in Y^{\prime}$, $\mathscr{I}(h_{m})|_{(X_{m})_{y}}=\mathcal{O}_{(X_{m})_{y}}$, we have $\mathscr{I}(h_{m}|_{D})=\mathcal{O}_{D}$, and the proof of the $\dim Y=1$ part is finished.

If $\dim Y>1$, a general element $D$ of the linear system $|f_{m}^{\ast}H|$ is a projective submanifold of $X$ by Bertini's theorem. Furthermore $h_{m}|_{D}$ is well-defined. The former computation shows that $(L_{m}|_{D},h_{m}|_{D})$ is a pseudo-effective line bundle over $D$ such that $\Theta_{h_{m}|_{D}}(L_{m}|_{D})\geqslant0$ and $\mathscr{I}(h_{m}|_{D})_{(X_{m})_{y}}=\mathcal{O}_{(X_{m})_{y}}$ for every $y\in D$. So by induction
\begin{equation*}
    (f_{m}|_{D})_{\ast}(K_{D}\otimes L_{m}|_{D})\otimes H^{\otimes\dim Y}|_{D}
\end{equation*}
is generically generated by its global sections. By adjunction formula, $K_{D}=K_{X_{m}}\otimes f^{\ast}_{m}H\otimes\mathcal{O}_{D}$, so another application of Theorem \ref{t9} shows that the restriction map
\begin{equation*}
   \begin{split}
       &H^{0}(X_{m},K_{X_{m}}\otimes L_{m}\otimes f^{\ast}_{m}H^{\otimes(\dim Y+1)}\otimes\mathscr{I}(h_{m}))\rightarrow\\
       &H^{0}(D,K_{D}\otimes L_{m}|_{D}\otimes f^{\ast}_{m}H^{\otimes\dim Y}|_{D}\otimes\mathscr{I}(h_{m}|_{D}))
   \end{split}
\end{equation*}
is surjective. The proof is complete.
\end{proof}

As we see in the proof, the ingredient is to use the Ohsawa--Takegoshi theorem to extend the section of $H^{0}(D,K_{D}\otimes L_{m}|_{D}\otimes f^{\ast}_{m}H^{\otimes\dim Y}|_{D}\otimes\mathscr{I}(h_{m}|_{D}))$ on a general divisor $D$, which requires that $h_{m}|_{D}$ as well as $h|_{f^{-1}(f_{m}(D))}$ is well-defined. This requirement is weakened to be an integral condition in our main result.

\begin{acknowledgment}
This paper was finished during the visit to Chalmers, so the author thanks Prof. Bo Berndtsson for many valuable discussions. Also the author wants to express his gratitude to his domestic supervisor Prof. Jixiang Fu for the support to this visit. Thanks also go to Ya Deng and Jian Xiao, who made a detailed introduction to the recent work of P\u{a}un--Takayama on the direct images.
\end{acknowledgment}

\address{

\small Current address: School of Mathematical Sciences, Fudan University, Shanghai 200433, People's Republic of China.

\small E-mail address: jingcaowu08@gmail.com, jingcaowu13@fudan.edu.cn
}

\end{document}